\documentclass[11pt]{amsart}

\usepackage{amsmath}
\usepackage{amsfonts}
\usepackage{amsthm}
\usepackage{amssymb}
\usepackage{times}

\newtheorem{theorem}{Theorem}[section]
\newtheorem{lemma}[theorem]{Lemma}

\theoremstyle{definition}
\newtheorem{definition}[theorem]{Definition}

\theoremstyle{remark}

\numberwithin{equation}{section}

\renewcommand{\vec}[1]{\boldsymbol{#1}}

\newcommand{\Z}{\mathbb Z}

\newcommand{\R}{\mathbb R}

\newcommand{\MNR}{{\mathcal R}}

\newcommand{\MV}{{\mathcal V}}

\newcommand{\GL}{\mathsf{GL}}

\newcommand{\gldr}{\GL_d(\R)}

\newcommand{\sd}{{\mathcal S}^d}
\newcommand{\sdm}{{\mathcal S}^{d,m}}
\newcommand{\sdo}{{\mathcal S}^d_{>0}}
\newcommand{\sdmo}{{\mathcal S}^{d,m}_{>0}}

\DeclareMathOperator{\trace}{trace}

\DeclareMathOperator{\vol}{vol}

\DeclareMathOperator{\cone}{cone}

\DeclareMathOperator{\Min}{Min}

\DeclareMathOperator{\relint}{relint}

\DeclareMathOperator{\grad}{grad}

\begin{document}

\title{Strict Periodic Extreme Lattices}

\author{Achill Sch\"urmann}
\address{Institute of Mathematics, University of Rostock, 18051 Rostock, Germany}
\email{achill.schuermann@uni-rostock.de}

\subjclass[2000]{Primary 52C17, Secondary 11H55}



\begin{abstract}
A lattice is called periodic extreme if it cannot locally be modified to yield
a better periodic sphere packing. It is called strict periodic extreme if
its sphere packing density is 
an isolated local optimum among periodic point sets.
In this note we show that a lattice is periodic extreme if and only if it is extreme,
that is, locally optimal among lattices. 
Moreover, we show that a lattice 
is strict periodic extreme if and only if it is extreme and non-floating.
\end{abstract}

\maketitle

\section{Introduction}

The {\em sphere packing problem}
asks for a non-overlapping arrangement of equally sized spheres, 
such that the fraction of space covered by spheres is maximized.
In dimension~$d$ equal to~$2$ and~$3$ optimal arrangements of spheres are given by
{\em lattices}, that is, by discrete subgroups of $\R^d$ (see~\cite{hales-2005}).
Optimal lattice sphere packings are known in dimensions $d\leq 8$ and $d=24$
(see \cite{cs-1998} and \cite{ck-2009}).
Although it seems highly likely, it is open whether or not non-lattice sphere packings with higher density exit for some $d\geq 4$.
All of the best known sphere packings so far are either given by lattices or   
{\em periodic point sets}, that is,
by a finite union of translates of a lattice.
These point sets are known to have densities 
that at least comes arbitrarily close to the density of best possible arrangements.

A natural idea to obtain good periodic sphere packings 
is to ``locally improve'' the best known lattice arrangements. 
Extending results in~\cite{schuermann-2010}, we show in this note that
a local improvement is not possible for locally optimal lattice sphere
packings (extreme lattices).
We show that extreme lattices yield even strict local sphere packing optima 
among periodic packings if they are {\em non-floating}. 
These packings do not allow simultaneous movements of spheres
(of some sublattice) without overlappings.

The paper is organized as follows. In Section~\ref{sec:recall}
we briefly recall some necessary terminology and results from~\cite{schuermann-2010} . 
We then extend these results In Section~\ref{sec:new}.

\section{A parameter space for periodic sets} \label{sec:recall}

An {\em $m$-periodic set} 
in $\R^d$ is the union of $m$ translates of a {\em lattice}~$L$
(a full rank discrete subgroup of $\R^d$): 
$$
\Lambda' = \bigcup_{i=1}^{m} \left( \vec{t}'_i + L \right)
,
$$
with translation vectors $\vec{t}'_i\in \R^d$, $i=1,\ldots ,m$.
The periodic set $\Lambda'$ is 
a linear image $\Lambda'=A\Lambda_{\vec{t}}$ of a {\em standard periodic set}
$$
\Lambda_{\vec{t}}
=
\bigcup_{i=1}^{m} \left( \vec{t}_i + \Z^d \right)
.
$$
Here, $A\in\gldr$ satisfies in particular $L=A\Z^d$.
Since we are only interested in properties of 
periodic sets up to isometries, 
we encode $\Lambda'$ by
the positive definite matrix $Q=A^tA$, 
together with the $m$ translation vectors $\vec{t}_1,\ldots, \vec{t}_m$.
Moreover,  without loss of generality
we may assume $\vec{t}_m=\vec{0}$.

We identify the $\binom{d+1}{2}$-dimensional
Euclidian vector space~$\sd$ of symmetric  $d\times d$ matrices
with the space of quadratic forms in $d$~variables.
The convex cone $\sdo$ of positive definite matrices is identified with the
space of {\em positive definite quadratic forms (PQFs)}.
For $m$-periodic point sets up to isometries we 
therefore consider the parameter space
$$
\sdmo = \sdo \times \R^{d\times (m-1)}
.
$$
For an extended discussion of this space and its symmetries we refer to
\cite{schuermann-2010} and \cite{cs-2011}.
Elements of $\sdmo$ are referred to as {\em periodic forms}.
Note that a given $m$-periodic set has infinitely many 
representations by periodic forms, as we may not only choose 
different $m$ (and therefore lattices~$L$), but we may also 
vary the basis~$A$ for the underlying lattice~$L$.

The space $\sdmo$ is contained in the Euclidean space  
$$
\sdm = \sd \times \R^{d\times (m-1)}
$$
with inner product $\langle \cdot , \cdot \rangle$, 
defined for $X=(Q,\vec{t})$ and $X'=(Q',\vec{t}')$ by
$$   \label{not:inner_product_sdm}
\langle X, X'\rangle = \langle Q, Q' \rangle + \sum_{i=1}^{m-1} \vec{t}_i^t \vec{t}'_i
.
$$
Here we use the usual inner product $\langle A, B \rangle = \trace (AB)$ on $\sd$.

The {\em generalized arithmetical minimum} of a periodic form $X=(Q,\vec{t})\in\sdmo$ 
is given by
$$   
\lambda(X) 
=
\min \{ 
Q[\vec{t}_i-\vec{t}_j-\vec{v}] : 
1\leq i,j \leq m \mbox{ and } \vec{v}\in \Z^d, \mbox{ with } \vec{v}\not=\vec{0} \mbox{ if } i=j \}
.
$$
It corresponds to the squared minimum distance among points of a periodic set represented by $X$.
The set $\Min X$ 
of {\em representations of the minimum} 
is the set of all 
$\vec{w}=\vec{t}_i-\vec{t}_j-\vec{v}$    
attaining $\lambda(X)$.
The sphere packing density of a periodic form~$X=(Q,\vec{t})$
and a corresponding $m$-periodic point set is
$$
\delta(X)
=
\left(\frac{\lambda (X)}{(\det Q)^{1/d}}\right)^{\frac{d}{2}} m \vol B^d / 2^d
.
$$

\begin{definition}
A periodic form~$X=(Q,\vec{t})\in \sdmo$ 
(and a corresponding periodic set represented by $X$)
is called {\em $m$-extreme} if it 
attains a local maximum of~$\delta$ within~$\sdmo$.
$X$ is called {\em isolated m-extreme}, if any sufficiently small change 
preserving $\lambda(X)$, necessarily lowers $\delta(X)$.
\end{definition}

For the study of $m$-extreme periodic forms in~$\sdmo$
we consider the {\em (generalized) Ryshkov set}
$$
\MNR_{m}
=
\left\{
X\in\sdmo : \lambda(X)\geq 1
\right\}
.
$$
Its boundary contains the periodic forms with 
generalized arithmetical minimum equal to~$1$.
This boundary is given by the linear inequalities
$$
p_{\vec{v}}(X) := Q[\vec{v}] = \langle X , (\vec{v}\vec{v}^t,0) \rangle \geq 1
$$
for $\vec{v}\in\Z^d\setminus\{\vec{0}\}$,  
and by the degree~$3$~polynomial inequalities
$$
p_{i,j,\vec{v}}(X)
:=
Q[\vec{t}_i-\vec{t}_j-\vec{v}] 
\geq 1
,
$$
for $i,j\in\{1,\ldots, m\}$ with $i\not=j$ and $\vec{v}\in\Z^d$. 
Our necessary and sufficient conditions of local sphere packing optimality
rely on a local linear approximation, based on the gradients 
\begin{equation}  \label{eqn:pijv-gradient}
(\grad p_{i,j,\vec{v}})(X) = 
(
\vec{w}\vec{w}^t,
\vec{0},\dots,\vec{0},
2Q\vec{w},
\vec{0},\dots,\vec{0},
-2Q\vec{w},
\vec{0},\dots,\vec{0}
)
.
\end{equation}
Here, $\vec{w}$ abbreviates $\vec{t}_i-\vec{t}_j-\vec{v}$
and the gradient should be understood as a 
vector in $\sdm=\sd\times \R^{d\times (m-1)}$
having an ``$\sd$-component'' 
$\vec{w}\vec{w}^t$
and a ``translational-component'' containing the zero-vector~$\vec{0}$ 
in all but the $i$th and $j$th column.
Because of the symmetry
$p_{i,j,\vec{v}}=p_{j,i,-\vec{v}}$ we may
restrict our attention to polynomials with $i\leq j$.
If $j=m$, the $j$th column is omitted and if $i=j$
the corresponding column is~$\vec{0}$.

The following definitions and the subsequent theorem from~\cite{schuermann-2010} 
generalize corresponding classical notions for lattices (see \cite{martinet-2003}, \cite{schuermann-2009}).

\begin{definition}
A periodic form~$X=(Q,\vec{t})\in \sdmo$ 
(and a corresponding periodic set represented by $X$)
is {\em $m$-perfect} if the {\em generalized Voronoi domain}
$$
\MV(X)
= \cone
\{(\grad p_{i,j,\vec{v}})(X) : \vec{t}_i-\vec{t}_j-\vec{v}\in \Min X 
    \mbox{ for some } \vec{v}\in\Z^d \}
$$
is full dimensional, that is, if $\dim \MV(X)= \dim \sdm = \binom{d+1}{2} + (m-1)d$. 

\medskip

A periodic form~$X=(Q,\vec{t})\in \sdmo$ 
(and a corresponding periodic set represented by $X$)
is {\em $m$-eutactic} if $(Q^{-1},\vec{0})$ 
is contained in the relative interior $\relint \MV(X)$ of~$\MV(X)$.
\end{definition}

\begin{theorem}  \label{thm:m-extreme-characterization}
If a periodic form~$X\in \sdmo$ is $m$-perfect and $m$-eutactic, then
$X$ is isolated $m$-extreme. 
\end{theorem}

\section{Characterizing strict periodic extreme sets} \label{sec:new}

In this section we derive characterizations of 
strict and weak local optimality of lattices among periodic sets
that are independent of~$m$ and the concrete realization as a periodic form.

\begin{definition}
A periodic point set is {\em (strict) periodic extreme} if it is (isolated) $m$-extreme for all possible representations~$X\in \sdmo$. 
\end{definition}

The following characterization of periodic extreme lattices is a strengthening
of Theorem~10 in~\cite{schuermann-2010}.

\begin{theorem} \label{thm:main-periodic}
A lattice 
is periodic extreme
if and only if it is extreme.
\end{theorem}

For the characterization of strict periodic extreme lattices we use 
the following definition, which goes back to Conway and Sloane (see~\cite{cs-1995}).

\begin{definition} \label{def:floating}
A periodic point set is called {\em floating} if there exists a representation
$$
\Lambda = \bigcup_{i=1}^{m} \left( \vec{t}_i + L \right)
$$
with a lattice~$L$
such that it is possible to continuously move a strict subset of the $m$~translates of~$L$, 
without lowering the minimum distance among elements in $\Lambda$.
\end{definition}

\begin{theorem} \label{thm:main-strict-periodic}
A lattice 
is strict periodic extreme
if and only if
it is extreme and non-floating.
\end{theorem}

Our proofs of these theorems rely on the 
following lemma, which is a strengthening of Lemma~9 in \cite{schuermann-2010}.

\begin{lemma} \label{lem:eutaxy}
Any representation $X\in\sdmo$ of a eutactic lattice (respectively PQF) is $m$-eutactic.
\end{lemma}

\begin{proof}
Let $Q\in\sdo$ be eutactic,
that is 
\begin{equation}  \label{eqn:eutaxy-condition}
Q^{-1} = \sum_{\vec{x}\in\Min Q} \alpha_{\vec{x}}\vec{x}\vec{x}^t
\end{equation}
for some choice of $\alpha_{\vec{x}} > 0$.

Let $X=(Q^X,\vec{t}^X)\in\sdmo$ be some representation of~$Q$, 
e.g. with $m>1$.
For a fixed $\vec{w}\in \Min X$ we define an abstract graph, 
whose vertices are the indices in $\{1,\ldots,m\}$.
Two vertices $i$ and $j$ are connected by an edge 
whenever there is some $\vec{v}\in \Z^d$ such that $\vec{w}=\vec{t}^X_i-\vec{t}^X_j-\vec{v}$.
This graph is a disjoint union of cycles (see the proof of Lemma~9 in~\cite{schuermann-2010} for details). 
So $\vec{w}$ induces a partition $(I_1,\dots,I_k)$ of $\{1,\ldots,m\}$.
Let~$I$ be an index set of this partition 
(containing the indices of a fixed cycle of the defined graph).
Summing over all triples $(i,j,\vec{v})$ with $i,j\in I$ and $\vec{v}\in\Z^d$ 
such that $\vec{w}=\vec{t}^X_i-\vec{t}^X_j-\vec{v}\in\Min X$,
we find (using~\eqref{eqn:pijv-gradient}):
$$
\sum_{ \genfrac{}{}{0pt}{2}{(i,j,\vec{v}) \in I^2\times\Z^d}{\mbox{\tiny with} \; \vec{v}=\vec{t}^X_i-\vec{t}^X_j-\vec{w}}} 
(\grad p_{i,j,\vec{v}})(X)
=
2|I| (\vec{w}\vec{w}^t, \vec{0})
.
$$
The factor $2$ comes from the symmetry 
$\grad p_{i,j,\vec{v}}=\grad p_{j,i,-\vec{v}}$.
Summation over all index sets $I$ of the partition yields
\begin{equation}   \label{eqn:summing_gradients}
\sum_{\genfrac{}{}{0pt}{2}{(i,j,\vec{v}) \in \{1,\ldots,m\}^2\times\Z^d}{\mbox{\tiny with} \; \vec{v}=\vec{t}^X_i-\vec{t}^X_j-\vec{w}}} 
(\grad p_{i,j,\vec{v}})(X)
= 2 m (\vec{w}\vec{w}^t, \vec{0})
.
\end{equation}

Each $\vec{w}\in \Min X$ corresponds to a unique $\vec{x}\in \Min Q$.
We set $\alpha_{\vec{w}} = \alpha_{\vec{x}}$ with $\alpha_{\vec{x}}$ from
the eutaxy condition~\eqref{eqn:eutaxy-condition}.
Multiplying~\eqref{eqn:summing_gradients} by $\alpha_{\vec{w}}/2m$ and
summing over all $\vec{w}\in \Min X$ yields
$$
(Q^{-1},\vec{0})
\; = \;
\sum_{\genfrac{}{}{0pt}{2}{\vec{w}\in\Min X,   (i,j,\vec{v}) \in \{1,\ldots,m\}^2\times\Z^d}{\mbox{\tiny with} \; \vec{v}=\vec{t}^X_i-\vec{t}^X_j-\vec{w} }}
(\alpha_{\vec{w}}/2m)
(\grad p_{i,j,\vec{v}})(X)
.
$$
Thus $X$ is $m$-eutactic.
\end{proof}

\begin{proof}[Proof of Theorem~\ref{thm:main-periodic}]
We can give a proof that is 
almost identical to the proof of Theorem~10 in~\cite{schuermann-2010}.
This Theorem states that a strongly eutactic and perfect lattice is periodic extreme.
By Lemma~\ref{lem:eutaxy} we can substitute
``strongly eutactic'' in its proof by ``eutactic''. 
Eutactic and perfect lattices, however,
are precisely the extreme lattices 
by a classical characterization of Voronoi~\cite{voronoi-1907}.
Thus we obtain that extreme lattices are periodic extreme.
The opposite implication follows from the definition.
\end{proof}

\begin{proof}[Proof of Theorem~\ref{thm:main-strict-periodic}]
We first translate Definition~\ref{def:floating} into the parameter space of 
periodic forms: $X\in \sdmo$ is called {\em floating} 
if there exists a purely translational change
$N=(0,\vec{t}^N)\not=0$ with 
$\lambda (X+\epsilon N)\geq \lambda (X)$ for $\epsilon$ 
on some intervall $[0,\epsilon_0]$ with $\epsilon_0>0$.

If a lattice (or PQF) is strict periodic extreme it is clearly extreme and non-floating.
We therefore consider an extreme lattice, respectively a PQF $Q\in\sdo$ 
which is non-floating, meaning it has no representation as a periodic form $X\in \sdmo$
that is floating.
Let $X=(Q^X,\vec{t}^X)\in\sdmo$ be a representation of $Q$.
As extreme PQFs $Q$ are in particular eutactic, we find by Lemma~\ref{lem:eutaxy} that $X$ is $m$-eutactic for any possible choice of~$m$.
If $X$ is also $m$-perfect in each case, 
we know by Theorem~\ref{thm:m-extreme-characterization}
that $X$ is also $m$-extreme for any possible choice of~$m$.
Hence, $Q$~would be strict periodic extreme.

So let us therefore assume that $X$ is not $m$-perfect.
By definition, the generalized Voronoi domain~$\MV(X)$ 
is not full dimensional in this case. 
As explained in \cite[Section 5]{schuermann-2010}, 
the assumption that $X$ is $m$-eutactic implies that
the only possible local changes $N\in\sdm$ of~$X$ 
that do not lower~$\lambda$ are of the form
\begin{equation} \label{eqn:N-not-0}
N=(Q^N,\vec{t}^N)\in \MV(X)^{\perp} \quad \mbox{with} \quad N\not=0
.
\end{equation}
By choosing $N$ in  $\MV(X)^{\perp}$ we in particular find
$$\langle N, (\grad p_{i,j,\vec{v}})(X) \rangle = 0$$
for all triples $(i,j,\vec{v})$ with $\vec{w}=\vec{t}^X_i-\vec{t}^X_j-\vec{v}\in \Min X$.
Using equation~\eqref{eqn:summing_gradients}, which we obtained in the proof of Lemma~\ref{lem:eutaxy},
we get $\langle N, (\vec{w}\vec{w}^t,\vec{0}) \rangle = Q^N[\vec{w}]=0$ 
for every fixed $\vec{w}\in \Min X$.
As $Q$ is perfect, the set 
$$
\left\{
\vec{w}\vec{w}^t
\; : \;
\vec{w}\in \Min X
\right\}
$$
has full rank $\binom{d+1}{2}$, implying $Q^N=0$.
So $N$ represents a purely translational change.
By the assumption that $Q$ is non-floating this is 
only possible for $N=0$, contradicting the choice of~$N$ in~\eqref{eqn:N-not-0}.
\end{proof}

%
%

\bibliographystyle{amsplain}

\providecommand{\bysame}{\leavevmode\hbox to3em{\hrulefill}\thinspace}
\providecommand{\MR}{\relax\ifhmode\unskip\space\fi MR }
\providecommand{\MRhref}[2]{%
  \href{http://www.ams.org/mathscinet-getitem?mr=#1}{#2}
}
\providecommand{\href}[2]{#2}

\end{document}